\documentclass[12pt]{amsart}
\usepackage[style=numeric,natbib=true,backend=bibtex,firstinits=true]{biblatex}
\bibliography{biblio}
\usepackage{amsmath}
\usepackage{amsthm}
\usepackage{amssymb}
\newtheorem{theorem}{Theorem}[section]

\newtheorem{definition}{Definition}

\newtheorem{lemma}[theorem]{Lemma}

\newtheorem{remark}[theorem]{Remark}

\newcommand\Tau{\mathrm{T}}
\newcommand\temp{\theta}
\newcommand\coll{u}
\newcommand\dep{v}

\newcommand\Diff{D}
\newcommand\Cond{K}
\newcommand\Soret{T}
\newcommand\Dufour{F}
\newcommand\Dca{A}
\newcommand\Dcb{B}
\newcommand{\Rdim}[1]{\mathbb{R}^#1}

\usepackage{xargs}
\usepackage{etoolbox}
\newcommandx{\inRange}[2][2=]{\ifstrequal{#2}{}
  {\in\{1,\ldots,#1\}}
  {\in\{#1,\ldots,#2\}}
}

\newcommand\Triangulation{\Tau_h}
\newcommand\fspace{S_h}
\newcommand{\ccoll}{\alpha}
\newcommand{\ctemp}{\beta}
\newcommand{\cdep}{\gamma}

\newcommand{\eltemp}{\tilde{\temp}^h}

\newcommand{\test}{\Phi}
\newcommand{\hsum}[1]{\sum_{#1=1}^{N_h}}
\newcommand{\Time}{I}
\usepackage{hyperref}
\usepackage[autostyle]{csquotes}
\makeatletter\def\namedlabel#1#2{\begingroup#2\def\@currentlabel{#2}\phantomsection\label{#1}\endgroup}\makeatother
\title[Error control of a thermo-diffusion system]{Error control  for the FEM approximation of an upscaled thermo-diffusion system with Smoluchowski interactions }
\author{Oleh Krehel} 
\address{Department of Mathematics and Computer Science, CASA -- Center for Analysis, Scientific computing and Applications, TU Eindhoven, The Netherlands}
\author{Adrian Muntean}

\address{Department of Mathematics and Computer Science, CASA -- Center for Analysis, Scientific computing and Applications,  ICMS -- Institute for Complex Molecular Systems, TU Eindhoven, The Netherlands}
\begin{document}
\maketitle

\begin{abstract} 
We analyze a coupled  system of evolution equations that describes the effect of thermal gradients on  the motion and deposition of $N$ populations
 of colloidal species diffusing and interacting together through Smoluchowski production terms.  This class of systems is particularly useful in studying drug delivery, contaminant transport in 
 complex media, as well as heat shocks thorough permeable media.  The particularity lies in the modeling of the nonlinear and nonlocal coupling between diffusion and thermal conduction. We investigate the semidiscrete as well as the fully discrete {\em a priori} error analysis of the finite elements approximation of 
the weak solution to a thermo-diffusion reaction system posed in a macroscopic domain. The mathematical techniques include energy-like estimates and compactness arguments.
\end{abstract}

\section{Introduction}

We are interested in quantifying  the effect of coupled  macroscopic fluxes\footnote{In this context,  the fluxes are driven by a suitable combination of heat and diffusion gradients \cite{groot1962non}.} on the aggregation, fragmentation and deposition of large populations of colloids traveling through a porous medium. To do so, we are using a  well-posed partly-dissipative coupled  system of  quasilinear parabolic equations posed in a connected open set  $\Omega$ with sufficiently smooth boundary.  The particular structure of the system has been obtained via periodic homogenization techniques in \cite{krehel2014thermo} [see e.g. Ref. \cite{hornung1991diffusion} for a methodological upscaling procedure of reactive flows through arrays of periodic microstructures]. 

The primary motivation of this paper is to develop and analyze appropriate numerical schemes to compute at macroscopic scales approximate solutions to our thermo-diffusion system with Smoluchowski interactions. Accounting for the interplay between heat, diffusion,
attraction-repulsion, and deposition of the colloidal particles is of
paramount importance for a number of applications including the
dynamics of the colloidal suspension in natural or man-made products
(e.g. milk, paints, toothpaste) \cite{Florea}, drug-delivery systems
\cite{andersson2004influences}, hierarchical assembly of biological
tissues \cite{Bart1}, group formation in actively interacting populations \cite{dark}, or heat stocks in porous materials  \cite{benevs2013global}. Further details  on colloids and their practical relevance are given in \cite{elimelech1998particle,krehel2012flocculation}, e.g. 

The discretizations shown in this paper have  been successfully used in \cite{krehel2014multiscale} to capture the effect of multiscale aggregation and deposition  mechanisms on the colloids dynamics traveling within a saturated porous medium in the absence of thermal effects. Now, we are preparing the stage to include the Soret and Dufour transport contributions -- cross-effects between diffusion and heat conduction; for more details on the macroscopic modeling of thermo-diffusion, we refer the reader to the monograph by De Groot and Mazur \cite{groot1962non}. The {\em a priori} estimates  are obtained in a similar fashion as for problems involving reactive flow in porous media (see, for instance, \cite{Florin,Eck_Knabner} and references cited therein), however specifics of the cross transport, interaction terms, and of the non-dissipative (ode) structure play here an important role and need to be treated carefully. For the numerical analysis of case studies in cross diffusion, we refer the reader for instance to \cite{Galiano,andreianov2011analysis} and \cite{Murakawa}. Note that there is not yet a unified mathematical approach to deal with general cross-diffusion or thermo-diffusion systems. Due to the presence of the nonlinearly coupled transport terms, essential difficulties arise in controlling the temperature gradients (and the gradients in the concentrations of colloidal populations) especially in more space dimensions (see e.g. \cite{benevs2013analysis}), the problem sharing many common features with the Stefan-Maxwell system for multicomponent mixtures (compare Refs. \cite{botheweierstrass,MS,Pruess} and the literature mentioned therein). 

In this paper, we investigate the semidiscrete as well as the fully discrete {\em a priori} error analysis of the finite elements approximation of 
the weak solution to a thermo-diffusion reaction system posed in a macroscopic domain that allows for aggregation, dissolution as well as deposition of colloidal species. The main results are summarized in Theorem \ref{ssdd} and Theorem \ref{FD}. The mathematical techniques used in the proofs include energy-like estimates and compactness arguments, exploiting the structure of the nonlocal coupling. Once these {\em a priori} estimates are proven and corrector estimates for the homogenization process explained in  \cite{krehel2014thermo} become available, then the next natural analysis step is to prepare a functional framework for the design optimally convergent MsFEM schemes approximating, very much in the spirit of \cite{Hou1999,Legoll2014}, multiscale formulations of our thermo-diffusion system.

The paper has the following structure: Section \ref{formulation} presents the setting of the model equations and briefly summarizes the meaning of the parameters and model components. We anticipate already at this point  the main results. 
In Section \ref{Technical}, we list the main mathematical analysis aspects of our choice of thermo-diffusion system and  briefly recall a collection of approximation theory results that are used in the sequel. Section \ref{semi} and Section \ref{fully} 
constitute the bulk of the paper. This is the place where we give the details of the proof of the semidiscrete and fully discrete {\em a priori} error  control, i.e. the proofs for Theorem \ref{ssdd} and Theorem \ref{FD}.  


\section{Formulation of the problem. Main results}\label{formulation}

Let $\Time$ denote an open sub-interval within the time interval $(0,T]$, and let $x\in\Omega$ be the variable pointing out the space position.
 The unknowns of the system are the temperature field $\theta$, the
mobile colloidal populations $\coll_i$ ($i\in\{1,\dots, N\}$), and the
immobile (already deposited) colloidal populations $v_i$
($i\in\{1,\dots, N\}$). $N\in\mathbb{N}$ represents the amount of the monomers in 
the largest colloidal species considered. All unknowns depend on both space and time variables $(x,t)\in\Omega\times\Time$. 

\begin{definition}
  Given $\delta>0$, we  introduce the mollifier:
  \begin{align}
    \label{eq:mollifier}
    &J_\delta(s):=
      \begin{cases}
        Ce^{1/(|s|^2-\delta^2)} & \text{if } |s| < \delta,\\
        0                   & \text{if } |s| \ge  \delta,
      \end{cases}
                              \intertext{where the constant $C>0$ is selected such that}
                              \nonumber
                                &\int_{\Rdim{d}}J_\delta=1,
  \end{align}
  see \cite{evans1998partial} for details.
\end{definition}

\begin{definition}
  Using $J_\delta$ from (\ref{eq:mollifier}), define the mollified gradient:
  \begin{equation}
    \label{eq:mollified-gradient}
    \nabla^\delta f:=\nabla
    \left[
      \int_{B(x,\delta)}J_\delta(x-y)f(y)dy
    \right],
  \end{equation}
  where $B(x,\delta)\subset {\mathbb{R}}^d$ is a ball centered in $x\in \Omega$ with radius $\delta$.
  \end{definition}
  
  With the Definition \ref{eq:mollified-gradient} at hand, 
  the following inequalities hold  for all $f\in L^\infty(\Omega)$ and  $g\in L^p(\Omega;\mathbb{R}^d)$ (with $1\le p\le \infty$):
  \begin{align}
    \label{eq:mollified-gradient-property}
    &\|\nabla ^\delta f \cdot  g\|_{L^p(\Omega)} \le  C \|f\|_{L^\infty{\Omega}} \|g\|_{L^p(\Omega;\mathbb{R}^d)}\\
    &\|\nabla ^\delta f\|_{\L^p(\Omega)} \le  C\|f\|_{L^2(\Omega)},
  \end{align}
  where the constant $C$ depends on the choice of the parameter $\delta$ and structure of the mollifier $J_\delta$.

For all $t\in \Time$,  the setting of our thermo-diffusion equations is the following: Find the triplet $(\temp,\coll_i,\dep_i)$ satisfying 
\begin{align}
  \label{eq:num-est.strong-formulation-beg}
  &\partial _t\temp + \nabla \cdot (-\Cond\nabla \temp)-\sum_{i=1}^N \Soret_i\nabla ^\delta \coll_i \cdot  \nabla \temp = 0 
  &&\text{in }\Omega\\
  \label{eq:num-est.strong-formulation-2}
  &\partial _t\coll_i + \nabla \cdot (-\Diff_i\nabla \coll_i) - \Dufour_i\nabla ^\delta\temp\cdot \nabla \coll_i +
    &&
 \\ 
&  + \Dca_i\coll_i - \Dcb_i\dep_i = R_i(\coll_i) &&  \text{in }\Omega\\
  \label{eq:num-est.strong-formulation-3}
  &\partial _t\dep_i = \Dca_i\coll_i - \Dcb_i\dep_i
  &&\text{in }\Omega\\
  &-\Cond\nabla \temp\cdot n=0 
  &&\text{on }\partial \Omega\\
  &\coll_i=0 
  &&\text{on }\partial \Omega\\
  &\temp(0,\cdot )=\temp^0(\cdot ) 
  &&\text{in }\Omega,\\
  &\coll_i(0,\cdot )=\coll_i^0(\cdot ) 
  &&\text{in }\Omega,\\
  \label{eq:num-est.strong-formulation-end}
  &\dep_i(0,\cdot )=\dep_i^0(\cdot ). 
  &&\text{in }\Omega.
\end{align} Here for all $i\in\{1,\dots,N\}$, the parameters $K$, $D_i$, $F_i$ and $T_i$ are effective transport coefficients for heat conduction, colloidal diffusion as well as Soret and Dufour effects. Furthermore, $A_i$ and $B_i$ are effective deposition coefficients. $\temp^0$ is the initial temperature profile, while $\coll_i^0$ and $\dep_i^0$ are the initial concentrations of colloids in mobile, and respectively, immobile state. General motivation on the ingredients of this system (particularly on Soret and Dufour effects) can be found in \cite{groot1962non}. Note that as direct consequence of fixing the threshold $N$, the system coagulates colloidal species (groups) until size $N$ only. 

This particular structure of the system has been derived in \cite{krehel2014thermo} by means of periodic homogenization arguments (two-scale convergence), scaling up the involved physicochemical processes from the pore scale (microscopic level, representative elementary volume (REV)) to a macroscopically  observable scale. 
\begin{remark}
Theorem 4.4 in \cite{krehel2014thermo} ensures the weak solvability of the system (\ref{eq:num-est.strong-formulation-beg})--(\ref{eq:num-est.strong-formulation-end}). Furthermore, under mild assumptions on the data and the parameters 
the weak solution is positive a.e. and satisfies a weak maximum principle. The basic properties of the weak solutions to  (\ref{eq:num-est.strong-formulation-beg})--(\ref{eq:num-est.strong-formulation-end}) are given in Section \ref{Technical}. 
\end{remark}

Denoting by $\temp^h(t)$ the continuous-in-time and
semidiscrete-in-space approximation of $\temp(t)$ and by $\temp^{h,n}$
the corresponding fully discrete approximation, with similar notation for the
other unknowns, we can formulate our main result: 
For all $t, t_n\in
\Time$, the following {\em a priori} estimates hold:
\begin{align}\label{i1}
  \| \temp^h(t)-\temp(t)\|  &+    \sum_{i=1}^N\| \coll_i^h(t)-\coll_i(t)\|  +    \sum_{i=1}^N\| \dep_i^h(t)-\dep_i(t)\|  && \nonumber\\
                        & \le C_1\| \temp^{0,h}-\temp^0\|+ C_2(\| \coll_i^{0,h}-\coll_i^0\| +\| \dep_i^{0,h}-\dep_i^0\|)+C_3h^2&&
\end{align}
and
\begin{align}\label{i2}
  &\| \temp^{h,n}-\temp^{n}\| +\sum_{i=1}^N\| \coll_i^{h,n}-\coll_i^{n}\|
    +\sum_{i=1}^N\| \dep_i^{h,n}-\dep_i^{n}\| \nonumber\\
  &\quad\le
    C_4\| \temp^{h,0}-\temp^0\|
    +C_5\left(\sum_{i=1}^N\| \coll_i^{h,0}-\coll_i^0\|
    +\sum_{i=1}^N\| \dep_i^{h,0}-\dep_i^0\|\right)\nonumber\\
  &\qquad
    +C_6(h^2+\tau).
\end{align}
The constants $C_1,\dots,C_6$ depend on data, but are independent of
the grid parameters $h$ and $\tau$. The hypotheses and the results under which (\ref{i1}) and (\ref{i2}) hold are stated in
Theorem \ref{ssdd} and Theorem \ref{FD}, respectively.  

The
following Sections focus exclusively on the proof of these
inequalities.

\section{Concept of weak solution. Technical preliminaries. Available results.}\label{Technical}

Our concept of weak solution is detailed as follows:
\begin{definition}\label{weak_cont}
  The triplet $(\temp,\coll_i,\dep_i)$ is a solution to (\ref{eq:num-est.strong-formulation-beg})-(\ref{eq:num-est.strong-formulation-end})
  if the following holds:
  \begin{align}
    \label{eq:num-est.weak-beg}
    &\begin{aligned}
      &\temp,\coll_i\in H^1(0,T;L^2(\Omega))\cap L^\infty(0,T;H^1(\Omega)),\\
      &\dep_i\in H^1(0,T;L^2(\Omega)),
    \end{aligned}
    \intertext{and for all $t\in J$ and $\phi\in H^1(\Omega):$} 
      & (\partial _t\temp,\phi)+(\Cond\nabla \temp,\nabla \phi)-\sum_{i=1}^N(\Soret_i\nabla ^\delta\coll_i\cdot \nabla \temp,\phi)=0,\\
    & (\partial _t\coll_i,\phi)+(\Diff_i\nabla \coll_i,\nabla \phi)
      -(\Dufour_i\nabla ^\delta\temp\cdot \nabla \coll_i,\phi)\nonumber\\
    &\qquad+(\Dca_i\coll_i-\Dcb_i\dep_i,\phi)=(R_i(\coll),\phi),\\
    \label{eq:num-est.weak-end}
    & (\partial _t\dep_i,\phi)=(\Dca_i\coll_i-\Dcb_i\dep_i,\phi).
  \end{align}
\end{definition}
 To be able to ensure the solvability of our thermo-diffusion problem, we assume that the following set of assumptions on the data (i.e. (A1)-(A2)) hold true:   
\begin{description}
\item[\namedlabel{eq:num-est.assumption-1}{\bf ($A_1$)}]
  $\Soret_i$, $\Dufour_i$, $\Dca_i$, $\Dcb_i$ are positive constants for 
  $i\inRange{N}$, and there exist $m$ and $M$ such that:
  $0 < m \le  \Cond \le  M$ and $0 < m \le  \Diff_i \le  M$.
  
\item[\namedlabel{eq:num-est.assumption-2}{\bf ($A_2$)}]
  $\temp^0\in L_+^\infty(\Omega)\cap H^2(\Omega), \coll_i^0\in L_+^\infty(\Omega)\cap H^2(\Omega), \dep_i^0\in L_+^\infty(\Gamma)$
  for $i\inRange{N}$.
\end{description}


Fix $h>0$ sufficiently small and let $\Triangulation$ be a
triangulation of $\Omega$ with
\begin{equation*}
 \max_{\tau\in \Triangulation} diam(\tau)\le h.
\end{equation*}
Let $\fspace$ denote the finite dimensional space of continuous
functions on $\Omega$ that reduce to linear functions in each of the
triangles of $\Triangulation$ and vanish on $\partial \Omega$.  Let
$\{P_j\}_{j=1}^{N_h}$ be the interior vertices of $\Triangulation$
with $N_h\in\mathbb{N}$. A function in $\fspace$ is then uniquely
determined by its values at the points $P_j$.  Let $\Phi_j$ be the
pyramid function in $\fspace$ which takes value $1$ at $P_j$, but
vanishes at the other vertices.  Then $\{\Phi_j\}_{j=1}^{N_h}$ forms a
basis for $\fspace$. Consequently, every $\varphi $ in $\fspace$ can be uniquely
represented as
\begin{equation}
  \label{eq:num-est.representation}
  \varphi(x)=\sum_{j=1}^{N_h}\alpha_j\Phi_j(x),\quad\text{with }\alpha_j:=\Phi(P_j), j\in\{1,\dots,N_h\},
\end{equation}
see e.g. Ref. \cite{knabner2003numerical}.

A smooth function $\mathcal{\sigma}$ defined on $\Omega$ which vanishes on $\partial \Omega$ can be 
approximated by its interpolant $I_h\mathcal{\sigma}$ in $\fspace$ defined as:
\begin{equation}
  \label{eq:num-est.interpolant}
  I_h\mathcal{\sigma}(x):=\sum_{j=1}^{N_h}\mathcal{\sigma}(P_j)\Phi_j(x).
\end{equation}

We denote below by $\| \cdot \| $ the norm of the space $L_2(\Omega)$ and by $\| \cdot \| _s$ that in the Sobolev
space $H^s(\Omega)=W_2^s(\Omega)$ with $s\in\mathbb{R}$. If $s=0$ we suppress the index.

We recall that for functions $v$ lying in $H_0^1(\Omega)$, the objects
$\| \nabla v\| $ and $\| v\| _1$ are equivalent norms. Let us also recall
Friedrichs' lemma (see, for instance,
\cite{brenner2008mathematical,ciarlet1978finite}): there exist
constants $c_F>0$ and $C_F>0$ (depending on $\Omega$, see
Ref. \cite{Michlin} for explicit expressions for these constants) such
that
\begin{equation}
  \label{eq:num-est.friedrichs}
  c_F\| \sigma\| _1\le C_F\| \nabla \sigma\| \le \| \sigma\| _1,\quad\forall \sigma\in H_0^1(\Omega).
\end{equation}

The following error estimates for the interpolant  $ I_h\mathcal{\sigma}$ of $\sigma$ [cf. (\ref{eq:num-est.interpolant})] are well-known
(see, e.g., \cite{brenner2008mathematical} or \cite{ciarlet1978finite}), namely for all
$\sigma\in H^2(\Omega)\cap H_0^1(\Omega)$ we have
\begin{align}
  \label{eq:num-est.interpolant-error-1}
  &\| I_h\sigma-\sigma\| \le Ch^2\| \sigma\| _2\\
  \label{eq:num-est.interpolant-error-2}
  &\| \nabla (I_h\sigma-\sigma)\| \le Ch\| \sigma\| _2.
\end{align}
Testing the equations (\ref{eq:num-est.strong-formulation-beg})-(\ref{eq:num-est.strong-formulation-2}) with
$\varphi \in\fspace$ leads to the following semi-discrete weak formulation of (\ref{eq:num-est.strong-formulation-beg})-(\ref{eq:num-est.strong-formulation-end})
 as given in Definition \ref{weak-discrete}.

\begin{definition}\label{weak-discrete}
  The triplet $(\temp^h,\coll_i^h,\dep_i^h)$ is a semidiscrete solution to (\ref{eq:num-est.strong-formulation-beg})-(\ref{eq:num-est.strong-formulation-end})
  if the following identities hold true for all $t\in \Time$ and $\varphi \in \fspace$:
  \begin{align}
    \label{eq:num-est.semi-weak-beg}
    &(\partial _t\temp^h,\varphi ) + (\Cond\nabla \temp^h,\nabla \varphi ) - \sum_{i=1}^N(\Soret_i\nabla ^\delta\coll_i^h\cdot \nabla \temp^h,\varphi )=0\\
    &(\partial _t\coll_i^h,\varphi ) + (\Diff_i\nabla \coll_i^h,\nabla \varphi ) - (\Dufour_i\nabla ^\delta\temp^h\cdot \nabla \coll_i^h,\varphi )\nonumber\\
    &\qquad+(\Dca_i\coll_i^h-\Dcb_i\dep_i^h,\varphi )=(R_i(\coll^h),\varphi )\\
    &(\partial _t\dep_i^h,\varphi ) = (\Dca_i\coll_i^h-\Dcb_i\dep_i^h,\varphi )\\
    &\temp^h(0)=\temp^{0,h}\\
    &\coll_i^h(0)=\coll_i^{0,h}\\
    \label{eq:num-est.semi-weak-end}
    &\dep_i^h(0)=\dep_i^{0,h}.
  \end{align}
Here, $\temp^{0,h}$, $\coll_i^{0,h}$, and $\dep_i^{0,h}$ are suitable 
  approximations of $\temp^0$, $\coll_i^0$, and $\dep_i^0$ respectively
  in the finite dimensional space $\fspace$.
\end{definition}
\begin{remark}
Note that $\dep_i$ as solution to (\ref{eq:num-est.strong-formulation-3}) can be expressed as:
\begin{align}
  \label{eq:num-est.dep-analytical}
  &\dep_i(t)=\left(\int_0^t \Dca_i \coll_i(s) e^{\Dcb_i s} ds\right) e^{-\Dcb_i t} + \dep_i^0 e^{-\Dcb_i t} &\text{for all }t\in \Time.
\end{align}
We will make this substitution later and also use
(\ref{eq:num-est.dep-analytical}) to obtain an error estimate for
$\dep_i^h$ based on the error estimate for $\coll_i^h$. This path can be followed due to the linearity of the equation. If the right-hand side of the ordinary differential equations becomes nonlinear, then a one-sided Lipschitz structure is needed to allow for the Gronwall argument to work.  
\end{remark}
\begin{remark} The existence of solutions in the sense of Definition \ref{weak_cont}  is ensured by periodic homogenization arguments in \cite{krehel2014thermo}, while the existence of solutions in the sense of Definition \ref{weak-discrete} follows by standard arguments. We omit to show the details of the existence proofs. Note that the existence of the respective solutions is nevertheless re-obtained here by straightforward compactness arguments. The proof of uniqueness of both kinds of solutions follows the lines of   \cite{krehel2014thermo}. 
\end{remark}
We represent the approximate solutions to the system (\ref{eq:num-est.strong-formulation-beg})--(\ref{eq:num-est.strong-formulation-end})  by means of the standard Galerkin Ansatz as:
\begin{align*}
 &\coll_i^h(x,t):=\hsum{j}\ccoll_{ij}(t)\test_j(x),\\
 &\temp^h(x,t):=\hsum{j}\ctemp_{j}(t)\test_j(x),\\
 &\dep_i^h(x,t):=\hsum{j}\cdep_{ij}(t)\test_j(x)
\end{align*}
for all $(x,t)\in\Omega\times \Time$. 
Based on the Galerkin projections, the semidiscrete model equations read:

\begin{align}
  \label{eq:num-est.semi-weak-exp-beg}
  &\hsum{j} \ctemp_{ij}'(t) (\test_j,\test_k) +
    \hsum{j} \ctemp_{ij}(\Cond_i\nabla \test_j,\nabla \test_k) \nonumber\\
  &\qquad- \sum_{i=1}^N\Soret_i \hsum{j} \hsum{l} \ctemp_{ij}(t) \ccoll_{il}(t) (\nabla ^\delta\test_l\cdot \nabla \test_j,\test_k)
    = 0\\
  \label{eq:num-est.semi-weak-exp-end}
  &\hsum{j} \ccoll_{ij}'(t) (\test_j,\test_k) +
  \hsum{j} \ccoll_{ij}(\Diff_i\nabla \test_j,\nabla \test_k) \nonumber\\
  &\qquad- \Dufour_i \hsum{j} \hsum{l} \ccoll_{ij}(t) \ctemp_l(t) (\nabla ^\delta\test_l\cdot \nabla \test_j,\test_k)
    = (R_i(\hsum{j}\ccoll_{ij}(t)\test_j),\test_k).
\end{align}

To abbreviate the writing of (\ref{eq:num-est.semi-weak-exp-beg})-(\ref{eq:num-est.semi-weak-exp-end}), we define:
  Define
\begin{align*}
  &\ccoll_i:=\ccoll_i(t)=(\ccoll_{i1}(t),\ldots,\ccoll_{i,N_h}(t))^T,\\
  &\ctemp:=\ctemp(t)=(\ctemp_{1}(t),\ldots,\ctemp_{N_h}(t))^T,\\
  &\cdep_i:=\cdep_i(t)=(\cdep_{i1}(t),\ldots,\cdep_{i,N_h}(t))^T,\\
  &G:=(g_{jk}),\:g_{jk}:=(\test_j,\test_k),\\
  &H_i^\coll:=(h^\coll_{ijk}),\:h^\coll_{ijk}:=(\Diff_i\nabla \test_j,\nabla \test_k),\\
  &H^\temp:=(h^\temp_{jk}),\:h^\temp_{jk}:=(\Cond\nabla \test_j,\nabla \test_k),\\
  &M:=(m_{jkl}),\:m_{jkl}:=(\nabla ^\delta\test_l\cdot \test_j,\test_k).
\end{align*}
Then (\ref{eq:num-est.semi-weak-exp-beg})-(\ref{eq:num-est.semi-weak-exp-end}) become:

\begin{equation}
  \label{eq:num-est.semi-weak-matrix}
  \begin{cases}
    &G\ctemp' + H^\temp\ctemp - \sum_{i=1}^N \Soret_i \ccoll_i^T M \ctemp = 0\\
    &G\ccoll_i' + H_i^\coll\ccoll_i - \Dufour_i \ctemp^T M \ccoll_i
    + G(\Dca_i\ccoll_i - \Dcb_i\cdep_i) \\ 
    & = (R_i(\hsum{j}\ccoll \test_j),\test_k)\\
    &G\cdep_i' = \Dca_iG\ccoll_i - \Dcb_iG\cdep_i\\
    &\ctemp(0)=\ctemp^0\\
    &\ccoll_i(0)=\ccoll_i^0\\
    &\cdep_i(0)=\cdep_i^0.
  \end{cases}
\end{equation}
Note that (\ref{eq:num-est.semi-weak-matrix}) is a nonlinear system of
coupled ordinary differential equations. Based on
\ref{eq:num-est.assumption-1}--\ref{eq:num-est.assumption-2}, we see
not only that $H^\temp$ and $H_i^\coll$ are positive definite, but
also that the right-hand side of the differential equations form a
global Lipschitz continuous function, fact which ensures the
well-posedness of the Cauchy problem
(\ref{eq:num-est.semi-weak-matrix}) on $\Time$ and eventually on its continuation on the whole interval  $(0,T]$; we refer the reader to \cite{Amann} for this kind of extension arguments for ordinary differential equations. Essentially, we get a unique solution
vector
 $$(\ctemp,\ccoll_i,\cdep_i)\in C^1(\bar\Time)^{N^h}\times C^1(\bar\Time)^{NN^h}\times C^1(\bar\Time)^{NN^h}$$ satisfying
(\ref{eq:num-est.semi-weak-matrix}); see \cite{muntean2010multiscale}
for the proof of the global Lipschitz property of the right-hand side
of a similar system of ordinary differential equations.

\section{Semi-discrete error analysis}\label{semi}

Our goal is to estimate the a priori error between the weak solutions of
(\ref{eq:num-est.semi-weak-beg})--(\ref{eq:num-est.semi-weak-end})
and the weak solutions of (\ref{eq:num-est.strong-formulation-beg})--(\ref{eq:num-est.strong-formulation-end}). We proceed very much in the spirit of Thome\'ee \cite{Thomee}; cf., for instance,  Chapter 13 and Chapter 14.  

We write the error as a sum of two terms:

\begin{equation}
  \label{eq:num-est.error-split}
  \temp^h-\temp = (\temp^h-\eltemp) + (\eltemp-\temp) = \psi + \rho.
\end{equation}

In (\ref{eq:num-est.error-split}),
$\eltemp$ is the elliptic projection in $\fspace$ of the exact solution $\temp$,
i.e. $\eltemp$ satisfies for all $t\geq 0$:
\begin{align}
  \label{eq:num-est.elliptic-projection}
  &(\Cond\nabla (\eltemp(t)-\temp(t)),\nabla \varphi )
    -\sum_{i=1}^N(\Soret_i\nabla ^\delta\coll_i\cdot \nabla (\eltemp(t)-\temp(t)),\varphi )=0 && 
\end{align}
for all  $\varphi \in \fspace$.

\begin{lemma}
  \label{lemma:num-est.lemma0}
  Let $k\in C^1(\bar{\Omega})$, $b\in L^\infty(\Omega,\Rdim{3})$, and $\nabla \cdot b\in L^\infty(\Omega)$.
  Suppose that $\gamma\in H_0^1(\Omega)$ is a weak solution to the elliptic boundary-value problem
  \begin{align}
    \label{eq:num-est.lemma0-1}
    -\nabla \cdot (k\nabla \gamma+b\gamma)=\delta\quad\text{in }\Omega,\quad
    \gamma=0\quad\text{on }\partial \Omega.
  \end{align}
  Additionally, assume 
  \begin{align}
    \label{eq:num-est.lemma0-2}
    \partial \Omega\in C^2.
  \end{align}
  Then we have
  \begin{align}
    \label{eq:num-est.lemma0-3}
    \| \gamma\| _2\le C\| \delta\| .
  \end{align}
\end{lemma}
\begin{proof}
  The proof of this result is a particular case of the proof of Theorem 4 given in \cite[][317]{evans1998partial}. We omit to repeat the arguments here. 
\end{proof}
\begin{remark}
  The condition (\ref{eq:num-est.lemma0-2}) can be relaxed to $\Omega$ being 
  a convex polygon, see \cite[][147]{grisvard2011elliptic} (compare Theorem 3.2.1.2 and Theorem 3.2.1.3).
\end{remark}

\begin{lemma}
  \label{eq:num-est.lemma0-alt}
  Let $k\in L^2(\Omega)$ and $b\in L^\infty(\Omega,\Rdim{3})$, and $k(x)\ge m>0$, and $m>\| b\| _\infty C_F$, where
  $C_F$ is the constant entering (\ref{eq:num-est.friedrichs}).
  Suppose that $\gamma\in H_0^1(\Omega)$ is a weak solution of the elliptic boundary-value problem
  \begin{align}
    \label{eq:num-est.lemma0-alt-1}
    -\nabla \cdot (k\nabla \gamma+b\gamma)=\delta\quad\text{in }\Omega,\quad
    \gamma=0\quad\text{on }\partial \Omega.
  \end{align}
  Then we have
  \begin{align}
    \label{eq:num-est.lemma0-alt-2}
    \| \gamma\| _2\le C\| \delta\|. 
  \end{align}
\end{lemma}
\begin{proof} We can directly verify that
  \begin{align*}
    &m\| \gamma\| ^2\le (k\nabla \gamma,\nabla \gamma)=(\delta,\gamma)+(b\cdot \gamma,\nabla \gamma)\\
    &\quad\le \| \delta\| \| \gamma\| +\| b\| _\infty\| \gamma\| \| \nabla \gamma\| \\
    &\quad\le \| \delta\| \| \gamma\| +\| b\| _\infty C_F\| \nabla \gamma\| ^2.
  \end{align*}
  Here, we used the Friedrichs inequality (\ref{eq:num-est.friedrichs}).
 Since $m>\| b\| _\infty C_F$, we have (\ref{eq:num-est.lemma0-alt-2}).
\end{proof}

\begin{lemma}
  \label{lemma:num-est.lemma1}
  Take $k\in L^\infty(\Omega)\cap H^1(\Omega)$ and $b\in L^\infty(\Omega,\Rdim{3})\cap H^1(\Omega,\Rdim{3})$ and assume that there exist $m$ and $M$ such that $0<m\le k(x)\le M$ for
  all $x\in \Omega$. Let $w\in H^2(\Omega)\cap H_0^1(\Omega)$ satisfying 
  \begin{align}
    \label{eq:num-est.lemma1-1}
    &(k\nabla (w_h-w),\nabla \varphi )-(b\cdot \nabla (w_h-w),\varphi )=0 && \mbox{ for all } \varphi \in S_h.
  \end{align}
  Then the following estimates hold:
  \begin{align}
    \label{eq:num-est.lemma1-2}
    &\| \nabla (w_h-w)\| \le C_1 h \| w\| _2\\
    \label{eq:num-est.lemma1-3}
    &\| w_h-w\| \le C_0 h^2 \| w\| _2.
  \end{align}
  Here, the constant $C_1$ depends on $\Triangulation$, $m$, and $M$.
  The constant $C_0$ depends additionally
  on the upper bound of $\nabla k$ and $b$ in the corresponding  $L^\infty$-norm.
\end{lemma}
\begin{proof}
 We proceed very much in the spirit of Ciarlet estimates.  By \ref{eq:num-est.assumption-1}, we have  that  
  \begin{align*}
    &m\| \nabla (w_h-w)\| ^2\le (k\nabla (w_h-w),\nabla (w_h-w))=\\
    &\quad (k\nabla (w_h-w),\nabla (w_h-\varphi )) + (k\nabla (w_h-w),\nabla (\varphi -w))=\\
    &\quad (b\cdot \nabla (w_h-w),w_h-\varphi )+(k\nabla (w_h-w),\nabla (\varphi -w))\le \\
    &\quad \| b\| _\infty \| \nabla (w_h-w)\| \| w_h-\varphi \|  + M\| \nabla (w_h-w)\| \| \nabla (\varphi -w)\|
  \end{align*}
  Take $\varphi :=I_hw$ - the Clement interpolant of $w$. Then we have:
  \begin{eqnarray}
    m\| \nabla (w_h-w)\| &\le& \| b\| _\infty(\| w_h-w\| +\| I_hw-w\| ) \nonumber\\
    &+ & M\| \nabla (I_hw-w)\| \le C_1h\| w\| _2,
  \end{eqnarray}
  which yields
  \begin{eqnarray}
    \label{eq:num-est.lemma1-4}
    \| \nabla (w_h-w)\| &\le & (C_1h+C_2\| b\| _\infty h^2)\| w\| _2\nonumber\\
    & + & \frac{\| b\| _\infty}{m}\| w_h-w\|. 
  \end{eqnarray}
 It is worth noting that  (\ref{eq:num-est.lemma1-4}) leads to (\ref{eq:num-est.lemma1-2}) when we show later that (at least) $$\| w_h-w\| \le Ch\| w\| _2.$$
  
  Next, we show (\ref{eq:num-est.lemma1-3}) using a  duality argument.
  Let $\gamma\in H_0^1(\Omega)$ solve the problem
  \begin{align*}
    -\nabla \cdot (k\nabla \gamma-b\gamma)=\delta\quad\text{in }\Omega,\quad
    \gamma=0\quad\text{on }\partial \Omega.
  \end{align*}
  Then 
  \begin{align*}
    &(w_h-w,\delta)=(w_h-w,-\nabla \cdot (k\nabla \gamma-b\gamma))\\
    &\quad=(k\nabla (w_h-w),\nabla \gamma)-(b\cdot \nabla (w_h-w),\gamma)\\
    &\quad=(k\nabla (w_h-w),\nabla (\gamma-\varphi ))-(b\cdot \nabla (w_h-w),\gamma-\varphi )\\
    &\quad+(k\nabla (w_h-w),\nabla \varphi )-(b\cdot \nabla (w_h-w),\varphi ).
  \end{align*}
  Let $\varphi :=I_h\gamma$ and use (\ref{eq:num-est.lemma1-1}):
  \begin{align*}
    (w_h-w,\delta)&\le M\| \nabla (w_h-w)\| \| \nabla (\gamma-I_h\gamma)\| \\
    &\quad+\| b\| _\infty\| \nabla (w_h-w)\| \| \gamma-I_h\gamma\| .
  \end{align*}
  Using the standard approximation properties for $I_h\gamma$, we get:
  \begin{align}
    \label{eq:num-est.lemma1-5}
    (w_h-w,\delta)\le (C_1Mh+C_2\| b\| _\infty h^2)\| \gamma\| _2\| \nabla (w_h-w)\| .
  \end{align}

  Using $\delta:=w_h-w$ in (\ref{eq:num-est.lemma1-5}), 
  and either Lemma \ref{lemma:num-est.lemma0} or Lemma \ref{eq:num-est.lemma0-alt},
  we obtain:
  \begin{align}
    \label{eq:num-est.lemma1-6}
    \| w_h-w\| \le (C_1Mh+C_2\| b\| _\infty h^2)C_3\| \nabla (w_h-w)\| .
  \end{align}
  Using (\ref{eq:num-est.lemma1-6}) in (\ref{eq:num-est.lemma1-4}) leads to:
  \begin{align}
    \label{eq:num-est.lemma1-7}
    \| \nabla (w_h-w)\| \le C_1h\| w\| _2 + C_2h\| \nabla (w_h-w)\| .
  \end{align}
  After solving the recurrence in (\ref{eq:num-est.lemma1-7}),
  (\ref{eq:num-est.lemma1-2}) is proven, and hence
  (\ref{eq:num-est.lemma1-3}) follows from (\ref{eq:num-est.lemma1-6}).
\end{proof}

\begin{lemma}
  \label{lemma:num-est.lemma2}
  Let $\eltemp$ be defined by (\ref{eq:num-est.elliptic-projection}),
  and let $\rho:=\eltemp-\temp$. Then the following estimates hold:
  \begin{align}
    \label{eq:num-est.lemma2-1}
    &\| \rho(t)\| +h\| \nabla \rho(t)\| \le C(\temp)h^2     &&t\in \Time,\\
    \label{eq:num-est.lemma2-2}
    &\| \rho_t(t)\| +h\| \nabla \rho_t(t)\| \le C(\temp)h^2 &&t\in \Time.
  \end{align}
\end{lemma}
\begin{proof}
  Using Lemma \ref{lemma:num-est.lemma1}, we have that $\| \nabla \rho\| \le C_1h\| \temp\| _2$ and
  $\rho\le C_0h^2\| \temp\| _2$, so (\ref{eq:num-est.lemma2-1}) follows by adding these estimates.

  To obtain (\ref{eq:num-est.lemma2-2}), we differentiate
  (\ref{eq:num-est.elliptic-projection}) with respect to time:
  \begin{align*}
    (k\nabla \rho_t,\nabla \varphi )-(b_t\cdot \nabla \rho+b\cdot \nabla \rho_t,\varphi )=0
  \end{align*}
  Assuming $k$ uniformly bounded, which it is, since it doesn't depend
  on $\theta$ in our case:
  \begin{align*}
    &m\| \nabla \rho_t\| ^2\le (k\nabla \rho_t,\nabla \rho_t)=(k\nabla \rho_t,\nabla (\eltemp_t-\varphi +\varphi -\temp_t))\\
    &\quad=(k\nabla \rho_t,\nabla (\varphi -\temp_t))+(k\nabla \rho_t,\nabla (\eltemp-\varphi ))\\
    &\quad=(k\nabla \rho_t,\nabla (\varphi -\temp_t))+(b_t\cdot \nabla \rho+b\cdot \nabla \rho_t,\eltemp-\varphi )
  \end{align*}
  We have used (\ref{eq:num-est.elliptic-projection}) in the last
  equation since $(\eltemp-\varphi)\in \fspace$.  Thus we get that 
  \begin{align*}
    m\| \nabla \rho_t\| ^2\le M\| \nabla \rho_t\| \| \nabla (\varphi -\temp_t)\| +(C_1(b)\| \nabla \rho\| +C_2(b)\| \nabla \rho_t\| )\| \eltemp-\varphi \| 
  \end{align*}
  Now, take $\varphi :=I_h\temp_t$ to obtain:
  \begin{align*}
    &m\| \nabla \rho_t\| \le M\| \nabla \rho_t\| Ch\| \temp_t\| _2+(C_1(b)\| \nabla \rho\| +C_2(b)\| \nabla \rho_t\| )(\| \rho_t\| +Ch\| \temp_t\| _2)\\
    &\quad\le \frac{m}{2}\| \nabla \rho_t\| ^2+Ch^2\| \temp_t\| _2^2+Ch(\| \rho_t\| +Ch\| \temp_t\| _2)\\
    &\qquad+C_2(u)\| \nabla \rho_t\| \| \rho_t\| +C_2(u)Ch\| \nabla \rho_t\| \| \temp_t\| _2.
  \end{align*}
  Using Young's inequality a few times, it finally follows that:
  \begin{align}
    \label{eq:num-est.lemma2-3}
    \| \nabla \rho_t\| ^2\le C_1h^2+C_2\| \rho_t\| ^2,
  \end{align}
  where $C_1$ and $C_2$ are independent of $h$.

  Now, we use the duality argument as in Lemma \ref{lemma:num-est.lemma1} to gain:
  \begin{align*}
    &(\rho_t,\delta)=(\rho_t,-\nabla \cdot (k\nabla \gamma-b\gamma))=(k\nabla \rho_t,\nabla \gamma)-(b\cdot \nabla \rho_t,\gamma)\\
    &\quad=(k\nabla \rho_t,\nabla (\gamma-\varphi ))-(b\cdot \nabla \rho_t,\gamma-\varphi )+(k\cdot \nabla \rho_t,\nabla \varphi )-(b\cdot \nabla \rho_t,\varphi )\\
    &\quad=(k\nabla \rho_t,\nabla (\gamma-\varphi ))-(b\cdot \nabla \rho_t,\gamma-\varphi ).
  \end{align*}
  Choosing $\varphi :=I_h\gamma$ and $\delta:=\rho_t$ yields
  \begin{align*}
    &\| \rho_t\| ^2\le C_1\| \nabla \rho_t\| (Mh+\| b\| _\infty h^2)\| \gamma\| _2\\
    &\quad\le C_2\| \nabla \rho_t\| (Mh+\| b\| _\infty h^2)\| \delta\| \le \\
    &\quad\le C_2\| \nabla \rho_t\| (Mh+\| b\| _\infty h^2)\| \rho_t\|.
  \end{align*}
  We now see that
  \begin{align}
    \label{eq:num-est.lemma2-4}
    \| \rho_t\| \le C(\coll,\temp)h\| \nabla \rho_t\| .
  \end{align}
  Combining (\ref{eq:num-est.lemma2-3}) and (\ref{eq:num-est.lemma2-4}) leads to convenient recurrence 
  relations, thus proving the statement of  the Lemma.
\end{proof}

\begin{lemma}
  \label{lemma:num-est.lemma3}
  Let $\eltemp$ be defined by (\ref{eq:num-est.elliptic-projection}).
  Then:
  \begin{align}
    \label{eq:num-est.lemma3-1}
    &\| \nabla \eltemp(t)\| _{\infty}\le C(\temp)&&t\in \Time.
  \end{align}
\end{lemma}
\begin{proof}
  We rely now on the inverse estimate:
  \begin{align}
    \label{eq:num-est.lemma3-2}
    &\| \nabla \varphi \| _{\infty}\le Ch^{-1}\| \nabla \varphi \| &&\forall \varphi \in \fspace
  \end{align}
  The statement (\ref{eq:num-est.lemma3-2}) is trivial to prove for
  linear approximation functions, since in this case $\nabla \varphi $ is constant
  on each triangle.  Using Lemma \ref{lemma:num-est.lemma2} and the
  known error estimate for $I_h\temp$, we have:
  \begin{eqnarray}
    \| \nabla (\eltemp-I_h\temp)\| _\infty &\le & Ch^{-1}\| \nabla (\eltemp-I_h\temp)\| \nonumber\\
    &\le & Ch^{-1}(\| \nabla \rho\| +\| \nabla (I_h\temp-\temp)\| )\le C(\temp).
  \end{eqnarray}

\end{proof}
The main result on the \emph{a priori} error control for the semi-discrete FEM approximation to our original system is given in the next Theorem. 
\begin{theorem}\label{ssdd}
  Let $(\temp,\coll_i,\dep_i)$ solve (\ref{eq:num-est.weak-beg})-(\ref{eq:num-est.weak-end}) and
  $(\temp^h,\coll_i^h,\dep_i^h)$ solve (\ref{eq:num-est.semi-weak-beg})-(\ref{eq:num-est.semi-weak-end}),
  and let assumptions \ref{eq:num-est.assumption-1}-\ref{eq:num-est.assumption-2} hold.
  Then the following inequalities hold:
  \begin{align}
    &\| \temp^h(t)-\temp(t)\| \le C\| \temp^{0,h}-\temp^0\| +C(\temp)h^2&&t\in \Time,\\
    &\| \coll_i^h(t)-\coll_i(t)\| \le C\| \coll_i^{0,h}-\coll_i^0\| +C(\coll_i)h^2&&t\in \Time,i\inRange{N}.
  \end{align}
\end{theorem}
\begin{proof}
  With an error splitting  as in (\ref{eq:num-est.error-split}), it is 
  enough to show a suitable upper bound for $\psi:=\temp^h-\eltemp$. We proceed in the following manner: 
  \begin{align*}
    &(\partial _t\psi,\varphi )+(\Cond\nabla \psi,\nabla \varphi )=(\partial _t\temp^h,\varphi )+(\Cond\nabla \temp^h,\nabla \varphi )-\sum_{i=1}^N (\Soret_i\nabla ^\delta \coll_i^h \cdot  \temp^h,\varphi )\\
    &\quad+\sum_{i=1}^N (\Soret_i\nabla ^\delta \coll_i^h \cdot  \temp^h,\varphi )-(\partial _t\eltemp,\varphi )-(\Cond\nabla \eltemp,\nabla \varphi )\\
    &\quad=-(\partial _t(\temp+\rho),\varphi )-(\Cond\nabla (\temp+\rho),\nabla \varphi )+\sum_{i=1}^N (\Soret_i\nabla ^\delta \coll_i^h \cdot  \temp^h,\varphi )\\
    &\quad=-(\partial _t\rho,\varphi )-(\Cond\nabla \rho,\nabla \varphi )+\sum_{i=1}^N(\Soret_i\nabla ^\delta\coll_i\cdot \nabla \rho,\varphi )\\
    &\qquad+\sum_{i=1}^N (\Soret_i(\nabla ^\delta \coll_i^h \cdot  \nabla \temp^h-\nabla ^\delta \coll_i \cdot  \nabla \temp-\nabla ^\delta\coll_i\cdot \nabla \rho),\varphi ).
  \end{align*}
  After eliminating the terms that vanish due to the definition of the
  elliptic projection, we obtain the following identity:
  \begin{align}
    \label{eq:num-est.theorem-1}
    &(\partial _t\psi,\varphi )+(\Cond\nabla \psi,\nabla \varphi )\nonumber\\
      &\quad=
        -(\partial _t\rho,\varphi )
        +\sum_{i=1}^N (\Soret_i(\nabla ^\delta \coll_i^h \cdot  \nabla \temp^h-\nabla ^\delta \coll_i \cdot  (\nabla \temp+\nabla \rho)),\varphi ).
  \end{align}
  We can deal with the second term on the right hand side of (\ref{eq:num-est.theorem-1}) as
  follows:
  \begin{align*}
    &\nabla ^\delta \coll_i^h \cdot  \nabla \temp^h-\nabla ^\delta \coll_i \cdot  \nabla \temp - \nabla ^\delta\coll_i\cdot \nabla \rho\\
    &\quad=
      (\nabla ^\delta\coll_i^h-\nabla ^\delta\coll_i)\cdot \nabla \temp^h+\nabla ^\delta\coll_i\cdot (\nabla \temp^h-\nabla \temp-\nabla \rho)\\
    &\quad=(\nabla ^\delta\coll_i^h-\nabla ^\delta\coll_i)(\nabla \psi+\nabla \eltemp)+\nabla ^\delta\coll_i\cdot \nabla \psi\\
  \end{align*}
  Now using $\varphi :=\psi$ as a test function and relying on the bound
  $$\| \nabla \eltemp\| _\infty<C(\temp)$$ (available cf. Lemma \ref{lemma:num-est.lemma3}), we obtain:
  \begin{align*}
    &\frac{1}{2}\partial _t\| \psi\| ^2+m\| \nabla \psi\| ^2\le \frac{1}{2}\| \partial _t\rho\| ^2+\frac{1}{2}\| \psi\| ^2\\
    &\quad+\sum_{i=1}^N(C\| u_i^h-u_i\| ^2+\varepsilon \| \nabla \psi\| ^2+\varepsilon \| u_i\| _\infty (\| \nabla \rho\| ^2+\| \nabla \psi\| ^2) + \| \psi\| ^2).
  \end{align*}
  Gronwall's inequality gives
  \begin{align*}
    \| \psi(t)\| ^2\le \| \psi(0)\| ^2+C\int_0^t (\| \partial _t\rho\| ^2+\| \nabla \rho\| ^2+\sum_{i=1}^N\| \coll_i^h-\coll_i\| ^2).
  \end{align*}
  The estimate
  \begin{align*}
    \| \psi(0)\| \le \| \temp^{h,0}-\temp^0\| +\| \eltemp(0)-\temp^0\| \le \| \temp^{h,0}-\temp^0\| +Ch^2\| \temp^0\| _2,
  \end{align*}
  together with the estimate $\| \coll_i^h-\coll_i\| \le C(\coll)h^2$ give the statement of the Theorem.
\end{proof}

\section{Fully discrete error analysis}\label{fully}

Let $\tau>0$ to be a small enough time step and use $t_n:=\tau n$ while
denoting $\temp^n:=\temp(t_n)$ and $\coll_i^n:=\coll_i(t_n)$. The
discrete in space approximations of $\temp^n$ and $\coll_i^n$ are
denoted as $\temp^{h,n}$ and $\coll_i^{h,n}$, respectively.


\begin{definition}
  The triplet $(\temp^{h,n},\coll_i^{h,n},\dep_i^{h,n})$ is a discrete solution to (\ref{eq:num-est.strong-formulation-beg})-(\ref{eq:num-est.strong-formulation-end})
  if the following identities hold for all $n\inRange{N}$ and $\varphi \in \fspace$:
  \begin{align}
    \label{eq:num-est.semi-weak-beg}
    &\frac{1}{\tau}(\temp^{h,n+1}-\temp^{h,n},\varphi )
      + (\Cond\nabla \temp^{h,n+1},\nabla \varphi )\nonumber\\
    &\qquad
      - \sum_{i=1}^N(\Soret_i\nabla ^\delta\coll_i^{h,n}\cdot \nabla \temp^{h,n+1},\varphi )=0,\\
    \label{eq:num-est.semi-weak-beg-1}
    &\frac{1}{\tau}(\coll_i^{h,n+1}-\coll_i^{h,n},\varphi )
      + (\Diff_i\nabla \coll_i^{h,n+1},\nabla \varphi )
      - (\Dufour_i\nabla ^\delta\temp^{h,n}\cdot \nabla \coll_i^{h,n+1},\varphi )\nonumber\\
    &\qquad
      + (\Dca_i\coll_i^{h,n+1}-\Dcb_i\dep_i^{h,n+1},\varphi )=(R_i(\coll^{h,n}),\varphi ),\\
    &\frac{1}{\tau}(\dep_i^{h,n+1}-\dep_i^{h,n},\varphi ) = (\Dca_i\coll_i^{h,n+1}-\Dcb_i\dep_i^{h,n+1},\varphi ),\\
    &\temp^{h,0}=\temp^{0,h},\\
    &\coll_i^{h,0}=\coll_i^{0,h},\\
    \label{eq:num-est.semi-weak-end}
    &\dep_i^{h,0}=\dep_i^{0,h}.
  \end{align}
Here, $\temp^{0,h}$, $\coll_i^{0,h}$, and $\dep_i^{0,h}$ are the
  approximations of $\temp^0$, $\coll_i^0$, and $\dep_i^0$ respectively
  in the finite dimensional space $\fspace$.
\end{definition}
\begin{remark}
  To treat (\ref{eq:num-est.semi-weak-beg}) and
  (\ref{eq:num-est.semi-weak-beg-1}), we use a semi-implicit
  discretization very much in the spirit of
  Ref. \cite{lakkis2013implicit}. Note however that other options for
  time discretization are possible.
\end{remark}
\begin{theorem}\label{FD}
  Let $(\temp,\coll_i,\dep_i)$ solve (\ref{eq:num-est.weak-beg})-(\ref{eq:num-est.weak-end}) and
  $(\temp^h,\coll_i^h,\dep_i^h)$ solve (\ref{eq:num-est.semi-weak-beg})-(\ref{eq:num-est.semi-weak-end}),
  and assumptions \ref{eq:num-est.assumption-1}-\ref{eq:num-est.assumption-2} hold.
  Then the following inequality holds:
  \begin{align}
    \label{eq:num-dis.estimate-all}
    &\| \temp^{h,n}-\temp^{n}\| +\sum_{i=1}^N\| \coll_i^{h,n}-\coll_i^{n}\|
      +\sum_{i=1}^N\| \dep_i^{h,n}-\dep_i^{n}\| \nonumber\\
    &\quad\le
      C_1\| \temp^{h,0}-\temp^0\|
      +C_2\sum_{i=1}^N\| \coll_i^{h,0}-\coll_i^0\|
      +C_3\sum_{i=1}^N\| \dep_i^{h,0}-\dep_i^0\| \nonumber\\
    &\qquad
      +C_4(h^2+\tau).
  \end{align}
  The constants  $C_1,\ldots,C_4$  entering (\ref{eq:num-dis.estimate-all}) depend on controllable norms of $\temp$, $\coll_i$, but are independent of $h$ and $\tau$.
\end{theorem}
\begin{proof}
  Similar with the methodology of the proof of the semidiscrete {\em a priori} error estimates, we split the error terms into two parts:
  \begin{align}
    \label{eq:num-dis.split-temp}
    &\temp^{h,n}-\temp^n=\rho^{\temp,n}+\psi^{\temp,n}:=(\temp^{h,n}-R_h\temp^n)+(R_h\temp^n-\temp^n),\\
    \label{eq:num-dis.split-coll}
    &\coll_i^{h,n}-\coll_i^n=\rho^{\coll_i,n}+\psi^{\coll_i,n}:=(\coll_i^{h,n}-R_h\coll_i^n)+(R_h\coll_i^n-\coll_i^n),
  \end{align}
  where $R_h\temp$ and $R_h\coll_i$ are the Ritz projections defined by:
  \begin{align}
    \label{eq:num-dis.ritz-temp}
    &(\Cond\nabla (R_h\temp-\temp),\nabla \varphi )=0,&&\forall \varphi \in \fspace,\\
    \label{eq:num-dis.ritz-coll}
    &(\Diff_i\nabla (R_h\coll_i-\coll_i),\nabla \varphi )=0,&&\forall \varphi \in \fspace,i\inRange{N}.
  \end{align}
  Here, $\psi^{\temp,n}$ and $\psi^{\coll_i,n}$ satisfy the following bounds:
  \begin{align}
    \label{eq:num-dis.elliptic-err-temp}
    &\| \psi^{\temp,n}\| \le Ch^2\| \temp^n\| _2,\\
    \label{eq:num-dis.elliptic-err-coll}
    &\| \psi^{\coll_i,n}\| \le Ch^2\| \coll_i^n\| _2,
  \end{align}
  so it remains to bound from above $\rho^{\temp,n}$ and $\rho^{\coll_i,n}$.
  We can write for $\rho^{\temp,n}$ the following identities:
  \newcommand\err[1]{\rho^{\temp,#1}}
  \newcommand\elerr[1]{\psi^{\temp,#1}}
  \begin{align*}
    &\frac{1}{\tau}(\err{n+1}-\err{n},\varphi )+(\Cond\nabla \err{n+1},\nabla \varphi )=
      \frac{1}{\tau}(\temp^{h,n+1}-\temp^{h,n},\varphi )+(\Cond\nabla \temp^{h,n+1},\nabla \varphi )\\
    &\qquad
      - \sum_{i=1}^N(\Soret_i\nabla ^\delta\coll_i^{h,n}\cdot \nabla \temp^{h,n+1},\varphi )
      + \sum_{i=1}^N(\Soret_i\nabla ^\delta\coll_i^{h,n}\cdot \nabla \temp^{h,n+1},\varphi )\\
    &\qquad
      -\frac{1}{\tau}(R_h\temp^{n+1}-R_h\temp^{n},\varphi )-(\Cond\nabla R_h\temp^{n+1},\nabla \varphi )\\
    &\quad
      = \sum_{i=1}^N(\Soret_i\nabla ^\delta\coll_i^{h,n}\cdot \nabla \temp^{h,n+1},\varphi )
      -\frac{1}{\tau}(R_h\temp^{n+1}-R_h\temp^{n},\varphi )-(\Cond\nabla \temp^{n+1},\nabla \varphi )\\
    &\quad
      = \sum_{i=1}^N(\Soret_i\nabla ^\delta\coll_i^{h,n}\cdot \nabla \temp^{h,n+1},\varphi )
      -\frac{1}{\tau}(R_h\temp^{n+1}-R_h\temp^{n},\varphi )\\
    &\qquad
      +(\partial _t\temp^{n+1},\varphi )
      -\sum_{i=1}^N(\Soret_i\nabla ^\delta\coll_i^{n+1}\cdot \nabla \temp^{n+1},\varphi ).
  \end{align*}
  After re-arranging the terms in the former expression, we obtain:
  \begin{align*}
    &\frac{1}{\tau}(\err{n+1}-\err{n},\varphi )+(\Cond\nabla \err{n+1},\nabla \varphi )\\
    &\quad
      =
      \underbrace{\sum_{i=1}^N(T_i(\nabla ^\delta\coll_i^{h,n}\cdot \nabla \temp^{h,n+1}-\nabla ^\delta\coll_i^{n+1}\cdot \nabla \temp^{n+1}),\varphi )}_A\\
    &\qquad
      +
      \underbrace{(\partial _t\temp^{n+1}-\frac{1}{\tau}(\temp^{n+1}-\temp^n),\varphi )}_B
      -
      \underbrace{\frac{1}{\tau}(\elerr{n+1}-\elerr{n},\varphi )}_C.
  \end{align*}
  Let us deal first with estimating the term $C$, then $B$, and finally, the term $A$. 
  
  To estimate the term $C$, we  use our semidiscrete estimate for $\| \partial _t\psi\| $ stated in Lemma \ref{lemma:num-est.lemma2},
  we get:
  \begin{align*}
    \| \frac{1}{\tau}(\elerr{n+1}-\elerr{n})\| =\| \frac{1}{\tau}\int_{t^n}^{t^{n+1}}\partial _t\psi^\temp\| \le C_C(\temp,\coll)h^2.
  \end{align*}
  The term $B$ can be estimated as follows:
  \begin{align*}
    B
    =(\frac{1}{\tau}\int_{t^n}^{t^{n+1}}(s-t^n)\partial _{tt}\temp(s)ds,\varphi )
    \le  \frac{\tau}{2} (\sup_{[t^n,t^{n+1}]}|\partial _{tt}\temp|)\| \varphi \|=C_B(\temp)\tau\| \varphi \| .
  \end{align*}
  Finally, to tackle the term $A$, we proceed as follows:
  \begin{align*}
    &A=(\nabla ^\delta\coll_i^{h,n}\cdot \nabla \temp^{h,n+1}-\nabla ^\delta\coll_i^{n+1}\cdot \nabla \temp^{n+1},\varphi )\\
    &\quad=(\nabla ^\delta\coll_i^{h,n}\cdot (\nabla \temp^{h,n+1}-\nabla \temp^{n+1})+\nabla \temp^{n+1}\cdot (\nabla ^\delta\coll_i^{h,n}-\nabla ^\delta\coll_i^{n+1}),\varphi )\\
    &\quad\le \varepsilon \| \coll_{i}^{h,n}\| _{\infty}(\| \nabla \rho^{n+1}\| ^2+\| \nabla \psi^{n+1}\| ^2)+C_{\varepsilon}\| \varphi \| ^2\\
    &\qquad
      +
      \underbrace{(\nabla \temp^{n+1}\cdot (\nabla ^\delta\coll_i^{h,n}-\nabla ^\delta\coll_i^{n+1}),\varphi )}_D.
  \end{align*}
  At its turn, the term $D$ can be expressed as:
  \begin{align*}
    &D=
      (\nabla \temp^{n+1}\cdot (\nabla ^\delta\coll_i^{h,n}-\nabla ^\delta\coll_i^{n}),\varphi )
      +
      (\nabla \temp^{n+1}\cdot (\nabla ^\delta\coll_i^{n}-\nabla ^\delta\coll_i^{n+1}),\varphi )\\
    &\quad\le \| \nabla \temp^{n+1}\| _\infty(\varepsilon \| \nabla ^\delta(\coll_i^{h,n}-\coll_i^n)\| ^2+C_\varepsilon \| \varphi \| ^2)
      + \underbrace{(\nabla \temp^{n+1}\cdot \int_{t^n}^{t^{n+1}}\partial _t\nabla ^\delta\coll_i,\varphi )}_E.\\
  \end{align*}
  Finally, the term $E$ can be estimated as:
  \begin{align*}
    E\le \| \nabla \temp^{n+1}\| _{\infty}\| \partial _t\nabla ^\delta\coll_i\| _\infty \tau \| \varphi \|.
  \end{align*}
  Adding together all the terms, and then substituting $\varphi :=\err{n+1}$ we finally obtain:
  \begin{align}
    \nonumber
    &\frac{1}{\tau}\| \err{n+1}\| ^2
      + m\| \nabla \err{n+1}\| ^2\le \frac{1}{\tau}\| \err{n}\| ^2
      + (C_B(\temp) \tau)^2\\
    \nonumber
    &\qquad
      + (C_C(\temp) h^2)^2
      + \varepsilon \| \coll_{i}^{h,n}\| _{\infty}(\| \nabla \err{n+1}\| ^2+\| \nabla \psi^{n+1}\| ^2)\\
    \label{eq:num-dis.theorem1-1}
    &\qquad
      + C_D\varepsilon \| \nabla ^\delta(\coll_i^{h,n}-\coll_i^n)\| ^2
      + (C_E(\coll,\temp)\tau)^2 + C \| \err{n+1}\| ^2\nonumber\\
    &\quad
      := C\| \err{n+1}\| ^2+R_n,
  \end{align}
  where the reminder $R_n$ is defined by:
  \begin{align*}
    R_n&:=\frac{1}{\tau}\| \err{n}\| ^2
      + (C_B(\temp) \tau)^2
      + (C_C(\temp) h^2)^2
      + \varepsilon \| \coll_{i}^{h,n}\| _{\infty}(\| \nabla \err{n+1}\| ^2+\| \nabla \psi^{n+1}\| ^2)\\
    &\qquad
      + C_D\varepsilon \| \nabla ^\delta(\coll_i^{h,n}-\coll_i^n)\| ^2
      + (C_E(\coll,\temp)\tau)^2
  \end{align*}

  For $R_n$ it holds:
  \begin{align*}
    R_n\le C(\temp,\coll) (h^2+\tau)^2.
  \end{align*}
  Note that we can derive a similar estimate for $\rho^{\coll_i,n+1}$, which we  then add to
  (\ref{eq:num-dis.theorem1-1}).
  
  To conclude, we denote $$e^n:=\| \err{n}\| ^2+\sum_{i=1}^N\| \rho^{\coll_i,n}\| ^2,$$
  to obtain the short structure
  \begin{align*}
    \frac{1}{\tau}e^{n+1}\le \frac{1}{\tau}e^n+C(e^{n+1} + R_n).
  \end{align*}
  From here it follows that:
  \begin{align*}
    (1-C\tau)e^{n+1}\le e^{n}+C\tau R_n.
  \end{align*}
  For sufficiently small $\tau$,  we can instead write the expression
  \begin{align*}
    e^{n+1}\le (1+C\tau)e^n+C\tau R_n.
  \end{align*}
  Iterating the later inequality, we obtain
  \begin{align*}
    e^{n+1}\le (1+C\tau)^{n+1}e^0+C\tau\sum_{j=1}^n R_j.
  \end{align*}
  Finally, this argument yields
  \begin{align*}
    e^{n+1}\le C\| \temp^{h,0}-\temp^0\| +C\| \coll_i^{h,0}-\coll_i^0\| +C(\temp,\coll)(h^2+\tau),
  \end{align*}
  which proves the Theorem \ref{FD}.
\end{proof}

\section*{Acknowledgments} We thank C. Venkataraman and O. Lakkis (both at Sussex) for very useful discussions on the fully discrete error control of multiscale parabolic systems and to I.S. Pop (Eindhoven) for the reading a preliminary version of these results.
AM and OK gratefully acknowledge the financial  support by the European Union through the Initial Training Network {\em Fronts and Interfaces in Science and Technology} of the Seventh Framework Programme (grant agreement number 238702). 

\ifdefined\included\else\printbibliography\end{document}\fi
\let\included\undefined